\documentclass[11pt]{amsart}

\usepackage{amssymb}

\newtheorem{theorem}{Theorem}[section]
\theoremstyle{plain}

\newtheorem{corollary}[theorem]{Corollary}

\newtheorem{example}[theorem]{Example}

\newtheorem{lemma}[theorem]{Lemma}

\newtheorem{problem}[theorem]{Problem}
\newtheorem{proposition}[theorem]{Proposition}

\numberwithin{equation}{section}

\title{Moufang loops with commuting inner mappings}
\author{G\'abor P. Nagy}
\email{nagyg@math.u-szeged.hu}
\address{SZTE Bolyai Institute, Department of Geometry,
Aradi v\'ertan\'uk tere 1, H-6720 Szeged, Hungary}

\author{Petr Vojt\v{e}chovsk\'{y}}
\email{petr@math.du.edu}
\address{Department of Mathematics, University of Denver, 2360 S Gaylord St,
Denver, Colorado, 80208, USA}



\DeclareMathOperator{\inn}{Inn}

\def\cl#1{\mathit{c\ell}(#1)}
\def\rank#1{\mathit{rank}(#1)}

\begin{document}

\begin{abstract}
We investigate the relation between the structure of a Moufang loop and its
inner mapping group. Moufang loops of odd order with commuting inner mappings
have nilpotency class at most two. $6$-divisible Moufang loops with commuting
inner mappings have nilpotency class at most two. There is a Moufang loop of
order $2^{14}$ with commuting inner mappings and of nilpotency class three.
\end{abstract}

\thanks{This paper was written during the first author's Marie Curie Fellowship MEIF-CT-2006-041105
at the University of W\"urzburg (Germany). The second author would like to
thank the University of W\"urzburg for providing a productive environment
during his visit.}

\maketitle

\section{Introduction}

Problems in loop theory are often attacked by tools of group theory, and it is
therefore a question of considerable interest in nonassociative algebra to
understand the relationship between the nilpotency class of a loop $Q$ and the
nilpotency class of its inner mapping group $\inn{Q}$.

When $Q$ is a group, the situation is transparent thanks to the textbook
isomorphism $Q/Z(Q)\cong\inn{Q}$. Hence for a nontrivial group $Q$ we have
\begin{equation}\label{eq:plus1}
    \cl{Q} = \cl{\inn{Q}} + 1.
\end{equation}
There are other varieties of loops, not necessarily associative, for which the
equality \eqref{eq:plus1} is satisfied, for instance in the case of commutative
Moufang loops \cite[Theorem VIII.11.5]{Bruck}.

Does \eqref{eq:plus1} hold for all loops? No, \eqref{eq:plus1} fails badly
already for loops of nilpotency class $3$, as there is a loop $Q$ of order $18$
with $\cl{Q}=3$ for which $\inn{Q}$ is not nilpotent. Such a loop was for the
first time constructed by Vesanen in 1995, and it can be found in \cite{KePh}.

But there are some good news. Bruck observed in  \cite{BruckTAMS} that
$\inn{Q}$ is abelian if $\cl{Q}=2$. Niemenmaa and Kepka showed in \cite{NiKe2}
that a finite loop $Q$ with $\inn{Q}$ abelian is nilpotent, without giving a
bound on $\cl{Q}$. More generally, a recent result of Niemenmaa \cite{Ni}
(based on \cite{Ma}) states that if $Q$ is a finite loop with $\inn{Q}$
nilpotent then $Q$ is nilpotent.

The converse of Bruck's observation does not hold. Using the technique of
connected transversals, Cs\"org\H{o} \cite{Csorgo} constructed a loop of order
$128$ and nilpotency class $3$ with abelian inner mapping group. We will
therefore call loops $Q$ with abelian $\inn{Q}$ satisfying $\cl{Q}>2$
\emph{loops of Cs\"org\H{o} type}.

How prevalent are loops of Cs\"org\H{o} type in loop theory? As we have already
mentioned, there are no loops of Cs\"org\H{o} type in the variety of
commutative Moufang loops. By \cite{CsDr}, there are no left conjugacy closed
loops of Cs\"org\H{o} type. On the other hand, numerous loops of Cs\"org\H{o}
type have been constructed in \cite{DrVo} by a combinatorial approach based on
symmetric trilinear forms. There is a Buchsteiner loop of Cs\"org\H{o} type of
order $128$ \cite{DrKi}.

But until now, no loop of Cs\"org\H{o} type has been found in a mainstream
variety of loops. In this paper we construct a Moufang loop of Cs\"org\H{o}
type of order $2^{14}$ and nilpotency class three. On the other hand, we show
that there is no uniquely $6$-divisible Moufang loop of Cs\"org\H{o} type and
no odd order Moufang loop of Cs\"org\H{o} type.

For a list of open problems concerning loops of Cs\"{o}rg\H{o} type, see
\cite{DrVo}.

Our method is mostly loop theoretical, similar to that of Bruck
\cite{Bruck}---we apply heavy commutator and associator calculus. It is known
that the category of Moufang loops is equivalent to the category of groups with
triality, however, this functorial equivalence is not fully understood, even
when restricted to solvable or nilpotent loops. A translation of the results
obtained here into the language of groups with triality could elucidate the
connection between (nilpotent) Moufang loops and (nilpotent) groups with
triality.

\section{Prerequisites}

\subsection{Notation}

For an element $x$ of a groupoid $Q$, denote by $L_x:Q\to Q$, $yL_x=xy$ the
\emph{left translation by $x$ in $Q$}, and by $R_x:Q\to Q$, $yR_x = yx$ the
\emph{right translation by $x$ in $Q$}.

\emph{Quasigroup} is a groupoid in which all translations are bijections. For
$x$, $y$ in a quasigroup $Q$, we denote by $x\backslash y$ the unique solution
$a\in Q$ to the equation $xa=y$, and by $y/x$ the unique solution $b\in Q$ to
the equation $bx=y$.

\emph{Loop} is a quasigroup $Q$ with \emph{neutral element} $1\in Q$, that is,
$x1=x=1x$ holds for every $x\in Q$. From now on, $Q$ always denotes a loop.

The \emph{inner mapping group} $\inn{Q}$ of $Q$ is the permutation group
generated by all \emph{middle inner mappings} (\emph{conjugations}) $T(x) =
R_xL_x^{-1}$, \emph{left inner mappings} $L(x,y) = L_xL_yL_{yx}^{-1}$, and
\emph{right inner mappings} $R(x,y) = R_xR_yR_{xy}^{-1}$, where $x$, $y\in Q$.

A subloop $H$ of $Q$ is \emph{normal}, $H\unlhd Q$, if $H\varphi = H$ for every
$\varphi\in\inn{Q}$. If $H$ is normal in $Q$, $Q/H$ is the factor loop defined
in the usual way.

The \emph{center} $Z(Q)=\{x\in Q;\;x\varphi = x$ for every $\varphi\in
\inn{Q}\}$ consists of all those elements of $Q$ that commute and associate
with all other elements of $Q$.

The \emph{iterated centra} of $Q$ are defined by $Z_0(Q)=1$, $Z_{i+1}(Q)/Z_i(Q)
= Z(Q/Z_i(Q))$. A loop $Q$ is \emph{nilpotent of class $m$}, $\cl{Q}=m$, if $m$
is the least integer satisfying $Z_m(Q)=Q$.

For $x$, $y\in Q$, let $[x,y]\in Q$ be the \emph{commutator} of $x$ and $y$,
defined by $xy=(yx)[x,y]$. For $x$, $y$, $z\in Q$, let $[x,y,z]\in Q$ be the
\emph{associator} of $x$, $y$ and $z$, defined by $(xy)z = (x(yz))[x,y,z]$.
Note that the commutator and associator are well-defined modulo $Z(Q)$, that
is, $[xc_1,yc_2] = [x,y]$, $[xc_1,yc_2,zc_3]=[x,y,z]$ for every $x$, $y$, $z\in
Q$ and $c_1$, $c_2$, $c_3\in Z(Q)$.

The \emph{associator subloop} $A(Q)$ of $Q$ is the least normal subloop of $Q$
such that $Q/A(Q)$ is a group. Hence $A(Q)$ is the least normal subloop of $Q$
containing all associators $[x,y,z]$. The \emph{commutator-associator subloop}
$Q'$, also called the \emph{derived subloop}, is the least normal subloop of
$Q$ such that $Q/Q'$ is an abelian group. Hence $Q'$ is the least normal
subloop of $Q$ containing all commutators $[x,y]$ and all associators
$[x,y,z]$.

Note that we have $\cl{Q}\le 2$ if and only if $Q/Z(Q)$ is an abelian group,
i.e., $Q'\le Z(Q)$, i.e., $[x,y]\in Z(Q)$ and $[x,y,z]\in Z(Q)$ for every $x$,
$y$, $z\in Q$.

The \emph{nucleus} of a loop $Q$ is the subloop $N(Q)=\{x\in Q;\; [x,y,z] =
[y,x,z] = [y,z,x]=1$ for every $y$, $z\in Q\}$.

We denote by $\langle S\rangle$ the subloop of $Q$ generated by $S\subseteq Q$,
and write $\langle x\rangle$ instead of $\langle\{x\}\rangle$, etc. For a loop
$Q$, let $\rank{Q}$ be the least cardinal $\kappa$ such that $Q$ is generated
by a subset of size $\kappa$.

Let $m>1$ be an integer. A loop $Q$ is \emph{$m$-divisible} if for every $x\in
Q$ there is $y\in Q$ such that $y^m=x$. A loop $Q$ is \emph{uniquely
$m$-divisible} if for every $x\in Q$ there is a unique $y\in Q$ such that
$y^m=x$, that is, if the power map $\varphi:Q\mapsto Q$, $z\mapsto z^m$ is a
bijection of $Q$. In such a case we denote the $m$-th root $x\varphi^{-1}$ of
$x$ by $x^{1/m}$. When $Q$ is finite, it is uniquely $m$-divisible if and only
if it is $m$-divisible.

A loop is said to be \emph{Moufang} if it satisfies the identity $x(y(xz)) =
((xy)x)z$. It is well-known that Moufang loops are \emph{diassociative}, that
is, any two elements generate an associative subloop. In particular, Moufang
loops are \emph{power-associative}, that is, every element generates an
associative subloop. This also means that for every element $x$ of a Moufang
loop $Q$ there is $x^{-1}\in Q$ such that $xx^{-1}=1=x^{-1}x$, and the inverses
satisfy $(xy)^{-1}=y^{-1}x^{-1}$, $x^{-1}(xy) = y$, $(xy)y^{-1}=x$, $[x,y] =
x^{-1}y^{-1}xy$, $[x,y]^{-1} = [y,x]$, and so on.

The famous Moufang theorem \cite{Moufang} states that three elements $x$, $y$,
$z$ of a Moufang loop generate an associative subloop if and only if
$[x,y,z]=1$, in which case $[y,x,z]=[x,z,y]=1$, too. Thus, a Moufang loop $Q$
satisfies $\cl{Q}\le 2$ if and only if $[[x,y],z]=[[x,y,z],u] =
[[x,y],z,u]=[[x,y,z],u,v]=1$ for every $x$, $y$, $z$, $u$, $v\in Q$.

From now on we will employ the \emph{dot-convention}, which uses $\cdot$ to
indicate priority of multiplication. For instance, the product $xy\cdot z$ is
to be read as $(xy)z$.

\subsection{Non-generators}

We say that $x\in Q$ is a \emph{non-generator} of $Q$ if $\langle S\rangle = Q$
whenever $\langle S\cup\{x\}\rangle = Q$.

It follows from \cite[Theorems 2.1 and 2.2]{Bruck} that in a finite nilpotent
loop $Q$ the derived subloop $Q'$ consists of non-generators. This allows us to
ignore commutators and associators in any generating subset $S$ of a finite
nilpotent loop $Q$, as long as $\langle S\rangle=Q$.

To illustrate the technique, if $Q$ is a finite nilpotent loop with
$\rank{Q}=3$ and $x$, $y$, $z$, $u$, $v\in Q$ then $\langle x, y, [x,z],
[y,u,v]\rangle$ is a proper subloop of $Q$. Of course, we are generally not
allowed to remove commutators and associators from sets $S$ that generate a
proper subloop of $Q$. For instance, we cannot conclude that $\langle x,
[x,y]\rangle$ has rank $1$ in the above example.

\subsection{Inner mappings and pseudo-automorphisms of Moufang loops}

A permutation $\varphi$ of a loop $Q$ is a \emph{pseudo-automorphism} if there
exists an element $c\in Q$ such that $(x\varphi)(y\varphi\cdot c) =
(xy)\varphi\cdot c$ for every $x$, $y\in Q$. The element $c$ is then referred
to as a \emph{companion} of $\varphi$.

Note that if $\varphi$ is a pseudo-automorphism with companion $c\in N(Q)$ then
$\varphi$ is an automorphism.

By \cite[Lemma VII.2.2]{Bruck}, the left inner mapping $L(x,y)$ has companion
$[y,x]$, the conjugation $T(x)$ has companion $x^3$, and $R(x,y) =
L(x^{-1},y^{-1})$ for all elements $x$, $y$ of a Moufang loop $Q$. Moreover, by
\cite[Lemma VII.5.4]{Bruck}, we have
\begin{displaymath}
    xL(z,y) = x[x,y,z]^{-1}
\end{displaymath}
in a Moufang loop.

\subsection{Center automorphisms}

An automorphism $\varphi$ of a loop $Q$ is called a \emph{center automorphism}
if for every $x\in Q$ we have $xZ(Q) = (x\varphi)Z(Q)$.

\begin{lemma}\label{lm:centaut}
Let $\varphi$ be a center automorphism of a loop $Q$. Define $\psi:Q\to Q$ by
$x\psi = x\backslash (x\varphi)$. Then $\psi$ is a homomorphism from $Q$ to
$Z(Q)$.
\end{lemma}
\begin{proof}
Since $xZ(Q) = (x\varphi)Z(Q)$, we have $x\backslash (x\varphi)\in Z(Q)$, and
$\psi$ is a mapping from $Q$ to $Z(Q)$. Now, $xy\cdot (xy)\psi= (xy)\varphi=
x\varphi\cdot y\varphi=x(x\psi)\cdot y(y\psi) = xy\cdot (x\psi\cdot y\psi)$,
where we have used $x\psi$, $y\psi\in Z(Q)$ in the last equality. Thus
$(xy)\psi = x\psi\cdot y\psi$.
\end{proof}

\subsection{The associated Bruck loop}

Let $Q$ be a uniquely $2$-divisible Moufang loop. Define $Q(1/2) = (Q,*)$ by
\begin{displaymath}
    x*y = x^{1/2}yx^{1/2}.
\end{displaymath}
Then, by \cite[Theorem VII.5.2]{Bruck}, $Q(1/2)$ is a power-associative loop
with the same identity and powers as $Q$, and $x*(y*(x*z)) = (x*(y*x))*z$ holds
in $Q(1/2)$. We call $Q(1/2)$ the \emph{Bruck loop associated with $Q$}.

\begin{lemma}\label{lm:commutativeBruck}
Let $Q$ be a uniquely $2$-divisible Moufang loop. Then $Q(1/2)$ is commutative
(hence a commutative Moufang loop) if and only if every subloop $H$ of $Q$ with
$\rank{H}\le 2$ satisfies $\cl{H}\le 2$.
\end{lemma}
\begin{proof}
Note that $\rank{H}\le 2$ implies that $H$ is a group. The following four
identities are equivalent: $x*y = x^{1/2}yx^{1/2}=y^{1/2}xy^{1/2} = y*x$,
$xy^2x=yx^2y$ (by unique $2$-divisibility), $[x,y]yx = yx[x,y]$, and
$[x,y]x=x[x,y]$. The first identity says that $Q(1/2)$ is commutative, while
the last identity says that every subloop $H$ with $\rank{H}\le 2$ satisfies
$\cl{H}\le 2$.
\end{proof}

The following result is \cite[Lemma VII.5.6]{Bruck}:

\begin{lemma}\label{lm:abelianBruck}
Let $Q$ be a uniquely $2$-divisible Moufang loop. Then $Q(1/2)$ is an abelian
group if and only if $[x,y]\in N(Q)$, $[x,y,z]\in Z(Q)$, and
$[[x,y],z]=[x,y,z]^2$ for every $x$, $y$, $z\in Q$.
\end{lemma}

\begin{corollary}\label{cr:groupBruck}
Let $Q$ be a uniquely $2$-divisible group with $\cl{Q}\le 2$. Then $Q(1/2)$ is
an abelian group.
\end{corollary}

In a uniquely $2$-divisible Moufang loop we have $(x^{1/2})^{-1} =
(x^{-1})^{1/2}$, and we denote this element by $x^{-1/2}$.

Recall that in a group of nilpotency class $2$ we have $[xy,z] =[x,z][y,z]$.

\begin{lemma}\label{lm:auxiliaryBruck}
Let $Q$ be a uniquely $2$-divisible Moufang loop with $\cl{Q}\le 2$. Let
$Q(1/2) = (Q,*)$ be the associated Bruck loop. Assume that $\cl{Q}\le 2$. Then:
\begin{enumerate}
\item[(i)] $y^{-1}xy = x*[x,y]$ for every $x$, $y\in Q$,
\item[(ii)] $[x^{1/2},y] = [x,y]^{1/2}$ for every $x$, $y\in Q$,
\end{enumerate}
where the commutators are calculated in $Q$.
\end{lemma}
\begin{proof}
Part (i) is equivalent to
\begin{displaymath}
    y^{-1}xy = x^{1/2}x^{-1}y^{-1}xyx^{1/2} = x^{-1/2}y^{-1}xyx^{1/2},
\end{displaymath}
or $[y^{-1}xy,x^{1/2}]=1$. But $[y^{-1}xy,x^{1/2}] =
[y^{-1},x^{1/2}][x,x^{1/2}][y,x^{1/2}] = [y^{-1}y,x^{1/2}] = 1$.

Part (ii) is equivalent to $[x,y] = [x^2,y]^{1/2}$, or $[x,y]^2 = [x^2,y]$,
which certainly holds.
\end{proof}

\section{Uniquely $2$-divisible Moufang loops}

Let $Q$ be a loop. We say that a mapping $f:Q^n\to Q$ is \emph{(multi)linear}
if the identities
\begin{align*}
    (x_1x_1',x_2,\dots,x_n)f &= (x_1,x_2,\dots,x_n)f\cdot (x_1',x_2,\dots,x_n)f,\\
    (x_1,x_2x_2',\dots,x_n)f &= (x_1,x_2,\dots,x_n)f\cdot (x_1,x_2',\dots,x_n)f,\\
                        &\vdots\\
    (x_1,x_2,\dots,x_nx_n')f &= (x_1,x_2,\dots,x_n)f\cdot (x_1,x_2,\dots,x_n')f
\end{align*}
are satisfied for every $x_1$, $x_1'$, $\dots$, $x_n$, $x_n'\in Q$.

For a loop $Q$ with two-sided inverses, a mapping $f:Q^n\to Q$ is
\emph{alternating} if for every $x_1$, $\dots$, $x_n\in Q$ and every
permutation $\pi$ of $\{1,\dots,n\}$ we have
\begin{displaymath}
    (x_{1\pi},\dots,x_{n\pi})f = ((x_1,\dots,x_n)f)^{\pi\,\mathrm{sgn}},
\end{displaymath}
where $\pi\,\mathrm{sgn}=1$ if $\pi$ is an even permutation, and
$\pi\,\mathrm{sgn}=-1$ if $\pi$ is an odd permutation.

For a uniquely $2$-divisible loop $Q$, it is easy to see that a linear mapping
$f:Q^n\to Q$ is alternating if and only if $(x_1,\dots,x_n)f$ vanishes whenever
$x_i=x_j$ for some $i\ne j$.

\begin{proposition}\label{pr:class2}
Let $Q$ be a uniquely $2$-divisible Moufang loop with $\cl{Q}\le
2$. Then:
\begin{enumerate}
\item[(i)] $x^3\in N(Q)$ for every $x\in Q$,
\item[(ii)] $[x,y,z]^3=1$ for every $x$, $y$, $z\in Q$,
\item[(iii)] the commutator mapping $Q^2\to Q$, $(x,y)\mapsto [x,y]$ is
alternating and linear,
\item[(iv)] the associator mapping $Q^3\to Q$, $(x,y,z)\mapsto [x,y,z]$ is
alternating and linear.
\end{enumerate}
\end{proposition}
\begin{proof}
By \cite[Theorem 3.3]{Hsu}, in every Moufang loop $H$ with $\cl{H}\le 2$ the
associator mapping is alternating and linear, and the following identity holds:
\begin{displaymath}
    [xy,z]=[x,z][y,z]\cdot [x,y,z]^3 = [x,z][y,z]\cdot [x^3,y,z].
\end{displaymath}
Furthermore, by \cite[Theorem A]{Hsu}, we have $x^6\in N(H)$. Since $Q$ is
uniquely $2$-divisible, so is $N(Q)$. As $x^3\cdot x^3 = x^6\in N(Q)$, it
follows that $x^3\in N(Q)$ and thus $[x^3,y,z]=1$. Hence the commutator mapping
in $Q$ is linear, and it is obviously alternating.
\end{proof}

\begin{lemma}\label{lm:class2}
Let $Q$ be a uniquely $2$-divisible Moufang loop with $\cl{Q}\le 2$. Then the
associators in $Q$ and $Q(1/2)$ agree.
\end{lemma}
\begin{proof}
We need to show that $(x(yz))^{-1}\cdot (xy)z = (x*(y*z))^{-1}*((x*y)*z)$.
Since $Q(1/2)$ is commutative by Lemma \ref{lm:commutativeBruck}, we can
rewrite the right hand side as
\begin{displaymath}
    (x^{-1}*(y^{-1}*z^{-1}))*(z*(y*x)).
\end{displaymath}
Our task is therefore to show that
\begin{equation}\label{eq:nasty}
    (x(yz))^{-1}\cdot (xy)z = u^{1/2}vu^{1/2},
\end{equation}
where
\begin{displaymath}
    u = x^{-1/2}(y^{-1/2}z^{-1}y^{-1/2})x^{-1/2},\quad
    v = z^{1/2}(y^{1/2}xy^{1/2})z^{1/2}.
\end{displaymath}
The commutator mapping is linear by Proposition \ref{pr:class2}, and all
commutators and associators are central by $\cl{Q}\le 2$. We can therefore
rewrite the commutator $[v,u]$ as
\begin{displaymath}
    [z(yx),x^{-1}(y^{-1}z^{-1})] =
    [z(yx),(x^{-1}y^{-1})z^{-1}] = [z(yx),(z(yx))^{-1}] = 1.
\end{displaymath}
Then $[v,u^{1/2}] = [v,u]^{1/2}=1$ by Lemma \ref{lm:auxiliaryBruck}(ii) and
unique $2$-divisibility, which means that we can interchange the factors
$u^{1/2}$, $v$ in the right hand side of \eqref{eq:nasty}. It remains to check
that
\begin{displaymath}
    (x(yz))^{-1}\cdot (xy)z = uv = x^{-1/2}(y^{-1/2}z^{-1}y^{-1/2})x^{-1/2}\cdot
    z^{1/2}(y^{1/2}xy^{1/2})z^{1/2}.
\end{displaymath}
Now,
\begin{align*}
    x^{-1/2}(y^{-1/2}z^{-1}y^{-1/2})x^{-1/2} &=
        x^{-1/2}(z^{-1}y^{-1})x^{-1/2}\cdot [y^{-1/2},z^{-1}]\\
        &=(z^{-1}y^{-1})x^{-1}\cdot [y^{-1/2},z^{-1}][x^{-1/2},z^{-1}y^{-1}]\\
        &=(x(yz))^{-1}\cdot [y^{-1/2},z^{-1}][x^{-1/2},z^{-1}y^{-1}]
\end{align*}
and, similarly,
\begin{displaymath}
    z^{1/2}(y^{1/2}xy^{1/2})z^{1/2} = (xy)z\cdot [y^{1/2},x][z^{1/2},xy].
\end{displaymath}
By Lemma \ref{lm:auxiliaryBruck}(ii) and the linear and alternating properties
of the commutator,
\begin{multline*}
    [y^{-1/2},z^{-1}][x^{-1/2},z^{-1}y^{-1}][y^{1/2},x][z^{1/2},xy]\\
    = [y,z]^{1/2}[x,z]^{1/2}[x,y]^{1/2}[y,x]^{1/2}[z,x]^{1/2}[z,y]^{1/2}=1.
\end{multline*}
\end{proof}

\begin{proposition}\label{pr:class3}
Let $Q$ be a uniquely $2$-divisible Moufang loop with $\cl{Q}\le 3$. Assume
further that every proper subloop $H$ of $Q$ has $\cl{H}\le 2$. Then the
mappings
\begin{align*}
    &Q^3\to Q,\quad (x,y,z)\mapsto [[x,y],z],\\
    &Q^4\to Q,\quad (x,y,z,u)\mapsto [[x,y,z],u]\\
    &Q^4\to Q,\quad (x,y,z,u)\mapsto [x,y,[z,u]],\\
    &Q^5\to Q,\quad (x,y,z,u,v)\mapsto [x,y,[z,u,v]]
\end{align*}
are linear. Moreover, if $\rank{Q}\ge 3$ then $(x,y,z)\mapsto [[x,y],z]$ is
alternating.
\end{proposition}
\begin{proof}
If $\rank{Q}\ge 3$ then $[[x,y],y]=1$, so $(x,y,z)\mapsto [[x,y],z]$ will be
alternating the moment it is linear.

When $Q$ is a group, the mapping $(x,y,z)\mapsto [[x,y],z]$ is linear, and
there is nothing else to prove. Thanks to diassociativity, we can assume that
$\rank{Q}\ge 3$.

Since $\cl{Q/Z(Q)}\le 2$, Proposition \ref{pr:class2} implies
$[xx',y]=[x,y][x',y]c_1$, $[x,yy']=[x,y][x,y']c_2$, $[xx',y,z] =
[x,y,z][x',y,z]c_3$, and so on, where $c_1$, $c_2$, $c_3\in Z(Q)$. Recall that
the commutator and the associator are well-defined modulo $Z(Q)$.

Thus $[[xx',y],z] = [[x,y][x',y]c_1,z] = [[x,y][x',y],z]$. Now, $H=\langle
[x,y],[x',y],z\rangle$ is a proper subloop of $Q$, and so $[[x,y][x',y],z] =
[[x,y],z][[x',y],z]$ by Proposition \ref{pr:class2} and $\cl{H}\le 2$. Using
analogous arguments, we establish the linearity of $[[x,y],z]$ and
$[[x,y,z],u]$ in all variables, and also the linearity of $[x,y,[z,u]]$ and
$[x,y,[z,u,v]]$ in the variables $z$, $u$, $v$.

Let us prove the linearity of $[x,y,[z,u]]$ in $x$. We have $xL(z,y) =
x[x,y,z]^{-1}$ in any Moufang loop. In particular, $xL([z,u],y) =
x[x,y,[z,u]]^{-1}$. Moreover, the companion of the pseudo-automorphism
$L([z,u],y)$ is the central element $[y,[z,u]]$, which means that $L([z,u],y)$
is an automorphism. It is in fact a center automorphism because $x\backslash
(xL([z,u],y)) = [x,y,[z,u]]^{-1}$ is central. By Lemma \ref{lm:centaut}, the
mapping $x\mapsto x\backslash (xL([z,u],y)) = [x,y,[z,u]]^{-1}$ is a
homomorphism, and this shows that $[x,y,[z,u]]$ is linear in $x$.

The linearity of $[x,y,[z,u,v]]$ in $x$ follows analogously, using the left
inner mapping $L([z,u,v],y)$ and the fact that $[y,[z,u,v]]$, $[x,y,[z,u,v]]\in
Z(Q)$.

Finally, the subloops $\langle x, y, [z,u]\rangle$, $\langle x, y,
[z,u,v]\rangle$ are proper in $Q$, the associator mapping is therefore
alternating on them by Proposition \ref{pr:class2}, and so $(x,y,z,u)\mapsto
[x,y,[z,u]]$ and $(x,y,z,u,v)\mapsto [x,y,[z,u,v]]$ are linear in $y$, too.
\end{proof}

\section{Moufang loops of odd order with commuting inner mappings}

In this section we prove our first main result:

\begin{theorem}\label{th:main}
Let $Q$ be a Moufang loop of odd order with $\inn{Q}$ abelian. Then $Q$ has
nilpotency class at most $2$.
\end{theorem}

Recall that a finite Moufang loop is of odd order if and only if it is
(uniquely) $2$-divisible.

Call $Q$ a \emph{minimal counterexample} to Theorem \ref{th:main} if $Q$ is a
uniquely $2$-divisible Moufang loop, $\inn{Q}$ is abelian, $\cl{Q}=3$,
$\cl{H}<3$ for every proper subloop $H$ of $Q$, and $\rank{Q}\ge 3$.

If Theorem \ref{th:main} does not hold, then there is indeed a minimal
counterexample to it, as defined above. To see this, consider any
counterexample $Q$ to Theorem \ref{th:main}, a $2$-divisible Moufang loop with
$\inn{Q}$ abelian such that $\cl{Q}>2$. Then $Q$ is nilpotent by \cite{NiKe2}
(we need $|Q|<\infty$ here), and upon replacing $Q$ with a suitable factor
loop, we can assume that $\cl{Q}=3$. Every strictly descending chain of
subloops of $Q$ of nilpotency class three has a minimal element, and upon
replacing $Q$ with that minimal element, we can assume that $\cl{H}<3$ for
every proper subloop $H$ of $Q$. Finally, Theorem \ref{th:main} holds for
groups, so we must have $\rank{Q}\ge 3$ by diassociativity.

\begin{lemma}\label{lm:small}
Let $Q$ be a Moufang loop and let $x$, $y$, $z\in Q$ be such that $T(x)L(y,z) =
L(y,z)T(x)$. Then $[[x,y,z],x]=1$.
\end{lemma}
\begin{proof}
Once again, $xL(z,y) = x[x,y,z]^{-1}$ in any Moufang loop. Thus $[x,y,z]^{-1}x
= (x[x,y,z]^{-1})T(x) = xL(z,y)T(x) = xT(x)L(z,y) = xL(z,y) = x[x,y,z]^{-1}$.
\end{proof}

The identity $[[x,y,z],x]=1$ plays a prominent role in the theory of Moufang
loops, as the following result of Bruck shows.

\begin{lemma}\label{lm:bruck}
Let $Q$ be a Moufang loop. Then $Q$ satisfies all or none of the following
identities:
\begin{enumerate}
\item[(i)] $[[x,y,z],x]=1$,
\item[(ii)] $[x,y,[y,z]]=1$,
\item[(iii)] $[x,y,z]^{-1} = [x^{-1},y,z]$,
\item[(iv)] $[x,y,z]^{-1} = [x^{-1},y^{-1},z^{-1}]$,
\item[(v)] $[x,y,z] = [x,zy,z]$,
\item[(vi)] $[x,y,z] = [x,z,y^{-1}]$,
\item[(vii)] $[x,y,z] = [x,xy,z]$.
\end{enumerate}
When these identities hold, the associator $[x,y,z]$ lies in the center of
$\langle x,y,z\rangle$, and the following identities hold for all integers $n$:
\begin{gather*}
    [x,y,z] = [y,z,x] = [y,x,z]^{-1},\\
    [x^n,y,z] = [x,y,z]^n,\\
    [xy,z] = [x,z][[x,z],y][y,z][x,y,z]^3,\\
    xL(y,z) = xR(y,z) = x[x,y,z].
\end{gather*}
\end{lemma}
\begin{proof}
The only observation not proved in \cite[Lemma VII.5.5]{Bruck} is
$xL(y,z)=xR(y,z) = x[x,y,z]$. But $xL(y,z)=x[x,z,y]^{-1}=x[x,y,z] =
x[x,y^{-1},z^{-1}] = xL(y^{-1},z^{-1}) = xR(y,z)$.
\end{proof}

\begin{lemma}\label{lm:xyzu}
Let $Q$ be a minimal counterexample to Theorem \ref{th:main}. Then
\begin{displaymath}
    [[x,y,z],u]=[[x,y],z,u]=1
\end{displaymath}
for every $x$, $y$, $z$, $u\in Q$.
\end{lemma}
\begin{proof}
By Lemma \ref{lm:small}, $[[x,y,z],x]=1$. By Lemma \ref{lm:bruck}, the
associator mapping is alternating. By Proposition \ref{pr:class3}, the mappings
$(x,y,z,u)\mapsto [[x,y,z],u]$, $(x,y,z,u)\mapsto [[x,y],z,u]$ are linear.
Parts (i) and (ii) of Lemma \ref{lm:bruck} then imply that the two mappings are
alternating.

If $\rank{Q}=3$, we conclude that $[[x,y,z],u]=[[x,y],z,u]=1$ for every $x$,
$y$, $z$, $u\in Q$, since the expressions are trivial on any $3$ elements. For
the rest of the proof, we can therefore assume that $\rank{Q}\ge 4$.

Then any three elements generate a proper subloop of nilpotency class at most
$2$. Consequently, by Proposition \ref{pr:class2}, $(x,y)\mapsto [x,y]$ is
linear, and we have
\begin{align*}
    [[x,y,z],u] &= [(x(yz))^{-1}\cdot (xy)z,u] =
    [x(yz),u]^{-1}[(xy)z,u]\\
    &=([x,u]\cdot [y,u][z,u])^{-1}([x,u][y,u]\cdot [z,u])
    =[[x,u],[y,u],[z,u]].
\end{align*}
Now, $\langle x, u, [y,u], [z,u]\rangle$ is a proper subloop of $Q$, so
$[[x,y,z],u] = [[x,u],[y,u],[z,u]]=1$.

Another proper subloop of $Q$ is $H=\langle x^{-1},xL(z,u),y\rangle$. The
restriction of $T(y)$ to $H$ is then an automorphism of $H$, since the
companion $y^3$ of $T(y)$ is nuclear by Proposition \ref{pr:class2}. Using
$xL(y,z)=x[x,z,y]^{-1}$, we calculate
\begin{align*}
    [xT(y),z,u]^{-1} &= (xT(y))^{-1}\cdot xT(y)L(u,z)
    = x^{-1}T(y)\cdot xL(u,z)T(y)\\
    &= (x^{-1}\cdot xL(u,z))T(y) = [x,z,u]^{-1}T(y),
\end{align*}
so $[xT(y),z,u] = [x,z,u]T(y)$.

Yet another proper subloop of $Q$ is $\langle x^{-1},[x,y],z,u\rangle = \langle
x^{-1},xT(y),z,u\rangle$. By Proposition \ref{pr:class2} and the above
calculation, we have
\begin{align*}
    [[x,y],z,u] &= [x^{-1}\cdot xT(y),z,u] = [x^{-1},z,u][xT(y),z,u]\\
    &=[x,z,u]^{-1}\cdot [x,z,u]T(y) = [[x,z,u],y] = 1.
\end{align*}
\end{proof}

\begin{lemma}\label{lm:xyz}
Let $Q$ be a minimal counterexample to Theorem \ref{th:main}. Then
\begin{displaymath}
    [[x,y],z]=1
\end{displaymath}
for every $x$, $y$, $z\in Q$.
\end{lemma}
\begin{proof}
Let $Q(1/2)=(Q,*)$ be the associated Bruck loop. By Lemma
\ref{lm:commutativeBruck}, $(Q,*)$ is a commutative Moufang loop. Let
$K=\langle [x,y],z,u\rangle$. By Lemma \ref{lm:xyzu}, $[[x,y],z,u]=1$, and the
Moufang theorem implies that $K$ is an associative subloop of $Q$, necessarily
with $\cl{K}\le 2$. Then Corollary \ref{cr:groupBruck} shows that $K(1/2)$ is
an abelian group. This means that $[x,y]\in Z(Q(1/2))$ for every $x$, $y\in Q$.

Calculating in $\langle x,y\rangle$, we have $xT(y) = x*[x,y]$, by Lemma
\ref{lm:auxiliaryBruck}(i), and thus
\begin{displaymath}
    xT(y)T(z) = x*[x,y]*[x*[x,y],z].
\end{displaymath}
The commutator mapping is linear in the group $\langle [x,y],x,z\rangle$, so
\begin{align*}
    [x*[x,y],z] &= [x^{1/2}[x,y]x^{1/2},z] = [x^{1/2},z][[x,y],z][x^{1/2},z]\\
    &=[x,z]^{1/2}[[x,y],z][x,z]^{1/2} = [x,z]*[[x,y],z],
\end{align*}
where we have used Lemma \ref{lm:auxiliaryBruck}(ii). Altogether,
\begin{displaymath}
    xT(y)T(z) = x * [x,y] * [x,z] * [[x,y],z].
\end{displaymath}
Since $T(y)T(z)=T(z)T(y)$, we deduce that $[[x,y],z]=[[x,z],y]$.

On the other hand, recall that the mapping $(x,y,z)\mapsto [[x,y],z]$ is
alternating by Proposition \ref{pr:class3}, hence
$[[x,y],z]=[[x,z],y]=[[x,y],z]^{-1}$, or $[[x,y],z]^2=1$. Then $[[x,y],z]=1$
follows by unique $2$-divisibility.
\end{proof}

\begin{lemma}\label{lm:xyzuv}
Let $Q$ be a minimal counterexample to Theorem \ref{th:main}. Then
\begin{displaymath}
    [[x,y,z],u,v]=1
\end{displaymath}
for every $x$, $y$, $z$, $u$, $v\in Q$.
\end{lemma}
\begin{proof}
By Lemmas \ref{lm:small}, \ref{lm:bruck} and Proposition \ref{pr:class3},
$(x,y,z,u,v)\mapsto [[x,y,z],u,v]$ is a linear mapping. Assume for a while that
$\rank{Q}=3$. Then any expression $[[x,y,z],u,v]$ using only the $3$ generators
of $Q$ will vanish (we can use Lemma \ref{lm:bruck} to cover the case
$[[x,y,z],x,y]=1$). For the rest of the proof, we can therefore assume that
$\rank{Q}>3$.

Then $\cl{\langle x,y,z\rangle}\le 2$, which means that the associators in $Q$
and $Q(1/2)$ agree, by Lemma \ref{lm:class2}. Let $L^*(y,z)$ denote the left
inner mapping in $Q(1/2)$. Then $xL^*(y,z) = x*[x,y,z]^{-1} =
x^{1/2}[x,y,z]^{-1}x^{1/2} = x[x,y,z]^{-1} = xL(y,z)$, because
$[x^{1/2},[x,y,z]]=1$. Since the left inner mappings in $Q$ and in the
commutative Moufang loop $Q(1/2)$ agree, $\inn{Q(1/2)}$ is abelian, and so
$\cl{Q(1/2)}\le 2$ by \cite[Theorem VII.11.5]{Bruck}. It follows that
$[[x,y,z],u,v] = 1$.
\end{proof}

Theorem \ref{th:main} now follows from Lemmas \ref{lm:xyzu}, \ref{lm:xyz} and
\ref{lm:xyzuv}.

\begin{problem} Let $Q$ be an infinite uniquely $2$-divisible Moufang with $\inn{Q}$ abelian. Is $\cl{Q}\le 2$?
\end{problem}

\section{Uniquely $6$-divisible Moufang loops with commuting inner mappings}

The finiteness assumption can be removed from Theorem \ref{th:main} in case of
$6$-divisible Moufang loops, cf. Theorem \ref{th:6}.

\begin{lemma}\label{lm:TxTyTyx}
Let $Q$ be a Moufang loop in which $[[x,y,z],x]=1$ holds for every $x$, $y$,
$z\in Q$. Then
\begin{displaymath}
    zT(y)T(x) = zT(yx)\cdot [x,y,z]^{-3}
\end{displaymath}
holds for every $x$, $y$, $z\in Q$.
\end{lemma}
\begin{proof}
By Lemma \ref{lm:bruck}, $[x,y,z]$ is central in $\langle x,y,z\rangle$. Note
that the following calculation takes place in $\langle x,y,z\rangle$, and that
every associator $[u,v,w]$ that appears below satisfies $\langle u,v,w\rangle =
\langle x,y,z\rangle$. Hence
\begin{align*}
    zT(y)T(x) &=x^{-1}(y^{-1}zy)x\\
        &=x^{-1}(y^{-1}z\cdot yx)[y^{-1}z,y,x]\\
        &=x^{-1}(y^{-1}\cdot z(yx))[y^{-1}z,y,x][y^{-1},z,yx]\\
        &=(x^{-1}y^{-1})z(yx)\cdot
            [y^{-1}z,y,x][y^{-1},z,yx][x^{-1},y^{-1},z(yx)]^{-1}\\
        &=zT(yx)\cdot [y^{-1}z,y,x][y^{-1},z,yx][x^{-1},y^{-1},z(yx)]^{-1}.
\end{align*}
Moreover, by Lemma \ref{lm:bruck},
\begin{align*}
    [y^{-1}z,y,x] & = [y^{-1},y^{-1}z,x ] = [y^{-1},z,x] = [x,y,z]^{-1},\\
    [y^{-1},z,yx] & = [y,yx,z] = [y,x,z] = [x,y,z]^{-1},\\
    [x^{-1},y^{-1},z(yx)] &=
        [x^{-1},y^{-1},(zy)x[z,y,x]^{-1}] = [x^{-1},y^{-1},(zy)x]\\
        &=[x^{-1},x^{-1}(y^{-1}z^{-1}),y^{-1}] = [x^{-1},y^{-1}z^{-1},y^{-1}]\\
        &=[y^{-1},y^{-1}z^{-1},x^{-1}]^{-1} = [y^{-1},z^{-1},x^{-1}]^{-1} = [x,y,z],
\end{align*}
and the result follows.
\end{proof}

\begin{lemma}\label{lm:TxyTyx}
Let $Q$ be a Moufang loop in which $[[x,y,z],x]=1$, all commutators are
nuclear, and $[[x,y],z] = [x,y,z]^2$ holds. Then
\begin{displaymath}
    zT(y)T(x) = zT(x)T(y)\cdot [x,y,z]^{-4}
\end{displaymath}
for every $x$, $y$, $z\in Q$.
\end{lemma}
\begin{proof}
By Lemma \ref{lm:TxTyTyx}, and since all commutators are in the nucleus, we
have
\begin{multline*}
    zT(yx) = zT([y^{-1},x^{-1}](xy)) \\
    = zT([y^{-1},x^{-1}])T(xy)
    \cdot [xy,[y^{-1},x^{-1}],z]^3 = zT([y^{-1},x^{-1}])T(xy).
\end{multline*}
By Lemma \ref{lm:bruck},
\begin{displaymath}
    zT([y^{-1},x^{-1}]) = z[z,[y^{-1},x^{-1}]]
    = z[[y^{-1},x^{-1}],z]^{-1} = z[y^{-1},x^{-1},z]^{-2} = z[x,y,z]^2.
\end{displaymath}
Since $[x,y,z]$ is central in $\langle x,y,z\rangle$ by Lemma \ref{lm:bruck},
we calculate
\begin{displaymath}
    zT(yx) = zT([y^{-1},x^{-1}])T(xy) = (z[x,y,z]^2)T(xy)
    = zT(xy)\cdot [x,y,z]^2.
\end{displaymath}
Finally, using the last equality and Lemmas \ref{lm:bruck}, \ref{lm:TxTyTyx},
we have
\begin{align*}
    zT(y)T(x) &= zT(yx)[x,y,z]^{-3} = zT(xy)[x,y,z]^{-1}
        = zT(xy)[x,y,z]^3\cdot [x,y,z]^{-4}\\
        &= zT(xy)[y,x,z]^{-3}\cdot [x,y,z]^{-4}
        = zT(x)T(y)\cdot [x,y,z]^{-4},
\end{align*}
as desired.
\end{proof}

\begin{theorem}\label{th:6}
Let $Q$ be a uniquely $6$-divisible Moufang loop with $\inn{Q}$ abelian. Then
$Q$ has nilpotency class at most $2$.
\end{theorem}
\begin{proof}
Every subloop $H$ of $Q$ with $\rank{H}=2$ is a group and hence satisfies
$\cl{H}\le 2$. By Lemma \ref{lm:commutativeBruck}, $Q(1/2)$ is a commutative
Moufang loop. By \cite[Lemma VII.5.7]{Bruck}, $x^3=1$ for every element $x$ in
the derived subloop of a commutative Moufang loop. Since $Q(1/2)$ contains no
elements of order $3$, we conclude that $Q(1/2)$ is an abelian group. By Lemma
\ref{lm:abelianBruck}, commutators are nuclear, associators are central, and
$[[x,y],z]=[x,y,z]^2$ in $Q$. Then Lemmas \ref{lm:small} and \ref{lm:TxyTyx}
yield $[x,y,z]^4=1$. Since $Q$ contains no elements of even order, we deduce
that $Q$ is a group and $\cl{Q}\le 2$.
\end{proof}

\section{Moufang $2$-loop of Cs\"org\H{o} type}\label{sc:2loop}

In this section, we show that Theorem \ref{th:main} cannot be extended to
Moufang loops of even order.

Throughout the section, let $X$ denote an associative ring and consider $X$ as
a natural $\mathbb Z$-module, that is, $nx=x+\cdots+x$ is well defined for all
$n\in \mathbb Z$ and $x\in X$.

Our construction is a generalization of a construction by R. H. Bruck. Indeed,
if $X$ is an algebra over field of characteristic $\neq 2$, then the
construction is precisely \cite[Example 3, p.128]{Bruck}. We mention that Bruck
did not establish any properties pertaining to the commutativity of the inner
mapping group.

\begin{proposition}\label{pr:bruckConstruction}
Let $X_1$ be an additive subgroup of $X$ such that $uu=0$ and $uv+vu=0$ holds
for every $u$, $v\in X_1$. For $n\ge 1$, denote by $X_n$ the additive subgroup
of $X$ generated by $u_1\cdots u_n$, where $u_i\in X_1$. Define multiplication
on the subspace $Q=X_1\times X_2\times X_3$ by
\begin{equation}\label{Eq:BruckF}
    (a,b,c)(a',b',c') = (a+a', b+b'+aa', c+c'+ba').
\end{equation}
Then $Q$ is a Moufang loop with neutral element $(0,0,0)$ and inverse
$(a,b,c)^{-1} = (-a,-b,-c+ba)$. Moreover,
\begin{align*}
    [(a,b,c),(a',b',c')] &= (0,2aa', ba'-b'a),\\
    [[(a,b,c),(a',b',c')],(a'',b'',c'')] &= (0,0,2aa'a''),\\
    [(a,b,c),(a',b',c'),(a'',b'',c'')] &= (0,0,aa'a''),\\
    (a,b,c)L((a',b',c'),(a'',b'',c'')) &= (a,b,c+aa'a''),\\
    (a,b,c)R((a',b',c'),(a'',b'',c'')) &= (a,b,c+aa'a''),\\
    (a,b,c)T((a',b',c')) &= (a,b+2aa',c+ab'+ba')
\end{align*}
hold for every $(a,b,c)$, $(a',b',c')$, $(a'',b'',c'')\in Q$.
\end{proposition}
\begin{proof}
We have $(a,b,c)(0,0,0)=(0,0,0)(a,b,c)=(a,b,c)$ for every $(a,b,c)\in Q$, so
$(0,0,0)$ is the neutral element of $Q$. Solving
$(a,b,c)(a',b',c')=(a'',b'',c'')$ for $(a,b,c)$ or $(a',b',c')$ leads to
\begin{align*}
    (a,b,c) &= (a''-a',b''-b'-a''a',c''-c'-(b''-c')a'),\\
    (a',b',c') &= (a''-a,b''-b-aa'',c''-c'b(a''-a)),
\end{align*}
respectively, where we have used the properties of $X_1$. Consequently,
\begin{displaymath}
    (a,b,c)^{-1} = (-a,-b,-c+ba).
\end{displaymath}
By straightforward calculation,
\begin{align*}
    (a,b,c)T((a',b',c'))&=(a',b',c')^{-1}((a,b,c)(a',b',c'))\\
    &=(-a',-b',-c'+b'a')(a+a',b+b'+aa',c+c'+ba')\\
    &=(a,b+2aa',c+ab'+ba').
\end{align*}
Let $x=(a,b,c)$, $y=(a',b',c')$, $z=(a'',b'',c'')$. Upon solving the equation
$xy = yx\cdot [x,y]$ for $[x,y]$, we get
\begin{displaymath}
    [(a,b,c),(a',b',c')] = (0,2aa',ba'-b'a),
\end{displaymath}
and upon solving the equation $(xy)z = x(yz)\cdot [x,y,z]$ for $[x,y,z]$, we
get
\begin{displaymath}
[(a,b,c),(a',b',c'),(a'',b'',c'')] = (0,0,aa'a'').
\end{displaymath}
It follows that $(0,0,c)\in Z(Q)$ for every $c\in X_3$, and $(0,b,c)\in N(Q)$
for every $(b,c)\in X_2\times X_3$. In particular, all commutators are in the
nucleus and all associators are in the center of $Q$.

We have $((xy)x)z = (xy)(xz)\cdot [xy,x,z] = ( x(y(xz))\cdot [x,y,xz])[xy,x,z]
= x(y(xz)) \cdot [x,y,xz][xy,x,z]$. Hence the Moufang identity $((xy)x)z =
x(y(xz))$ holds if and only if $[x,y,xz][xy,x,z]$ vanishes. With $x$, $y$, $z$
as above, we have $[x,y,xz][xy,x,z] = (0,0,aa'(a+a''))(0,0,(a+a')aa'') =
(0,0,aa'a'')(0,0,a'aa'') = (0,0,aa'a''+a'aa'') = (0,0,0)$.

Lemma \ref{lm:bruck} yields $xL(y,z) = xR(y,z) = x[x,y,z]$, finishing the proof.
\end{proof}

\begin{proposition}\label{pr:innQchar4}
Let $Q=X_1\times X_2\times X_3$ be the Moufang loop constructed in Proposition
\ref{pr:bruckConstruction}, with the same assumptions on $X$. Let $H$ be
defined on $X_1\times X_2$ with multiplication
\begin{displaymath}
    (u,v)(u',v') = (u+u',v+v'+2uu').
\end{displaymath}
Then $H$ is a group and $\inn{Q}=\{T_a;\;a\in Q\}\cong H$. Moreover, if
$4X=\{4x;\;x\in X\}=0$ then $\inn{Q}$ is an abelian group of exponent $4$.
\end{proposition}
\begin{proof}
It is easy to see that $H$ is associative, has neutral element $(0,0)$, and
$(u,v)^{-1} = (-u,-v+2uu) = (-u,-v)$. Moreover,
\begin{align*}
    [(u,v),(u',v')] &= (u,v)^{-1}(u',v')^{-1}(u,v)(u',v')\\
        &=(-u-u',-v-v'+2uu')(u+u',v+v'+2uu')\\
        &=(0,4uu'-2(u+u')(u+u')) = (0,4uu'),
\end{align*}
and $(u,v)^4 = (2u,2v)^2 = (4u,4v)$. Hence $H$ is an abelian group of exponent
$4$ when $4X=0$.

For  $(u,v)\in H$ define the map $S(u,v):Q\to Q$ by
\begin{displaymath}
(a,b,c)S(u,v) = (a,b+2au,c+av+bu).
\end{displaymath}
As
\begin{align*}
    (a,b,c)S(u,v)S(u',v') &= (a,b+2au,c+av+bu)S(u',v')\\
        &=(a,b+2au+2au',c+av+bu+av'+(b+2au)u')\\
        &=(a,b+2a(u+u'),c+a(v+v'+2uu')+b(u+u'))\\
        &=(a,b,c)S(u+u',v+v'+2uu'),
\end{align*}
the set $\{S(u,v);\; (u,v)\in X_1\times X_2\}$ is isomorphic to $H$. By
Proposition \ref{pr:bruckConstruction}, we have
\begin{displaymath}
    L((a',b',c'),(a'',b'',c'')) = R((a',b',c'),(a'',b'',c'')) = S(0,a'a'')
\end{displaymath}
and
\begin{displaymath}
    T((a',b',c')) = S(a',b').
\end{displaymath}
This shows at once that $\inn{Q}$ is isomorphic to $H$, and that it consists of
conjugations of $Q$.
\end{proof}

We now construct a ring $X$ for which Proposition \ref{pr:bruckConstruction}
yields a Moufang loop $Q$ of Cs\"org\H{o} type.

Put $R=\mathbb Z_4$, $X=R^7$ and let $\{e_1,\ldots,e_7\}$ be a set of free
generators of the $R$-module $X$. Define the multiplication on the generators
according to
\begin{displaymath}
    \begin{array}{c|ccccccc}
    &e_1&e_2&e_3&e_{4}&e_{5}&e_{6}&e_{7}\\
    \hline
    e_1&0&e_{4}&e_{5}&0&0&e_{7}&0\\
    e_2&-e_{4}&0&e_{6}&0&-e_{7}&0&0\\
    e_3&-e_{5}&-e_{6}&0&e_{7}&0&0&0\\
    e_{4}&0&0&e_{7}&0&0&0&0\\
    e_{5}&0&-e_{7}&0&0&0&0&0\\
    e_{6}&e_{7}&0&0&0&0&0&0\\
    e_{7}&0&0&0&0&0&0&0
    \end{array}
\end{displaymath}
and extend it to $X$ by $R$-linearity. This turns $X$ into an associative ring
satisfying $4X=0$. In order to verify associativity, we observe that all
products $e_i(e_je_k)$ vanish except when $\{i,j,k\}=\{1,2,3\}$, and then
\begin{eqnarray*}
    e_1(e_2e_3)\!=\!e_1e_6\!=\!e_7\!=\!e_4e_3\!=\!(e_1e_2)e_3, &&
    e_1(e_3e_2)\!=\!-e_1e_6\!=\!-e_7\!=\!e_5e_2\!=\!(e_1e_3)e_2, \\
    e_2(e_1e_3)\!=\!e_2e_5\!=\!-e_7\!=\!-e_4e_3\!=\!(e_2e_1)e_3, &&
    e_2(e_3e_1)\!=\!-e_2e_5\!=\!e_7\!=\!e_6e_1\!=\!(e_2e_3)e_1, \\
    e_3(e_1e_2)\!=\!e_3e_4\!=\!e_7\!=\!-e_5e_2\!=\!(e_3e_1)e_2, &&
    e_3(e_2e_1)\!=\!e_3e_4\!=\!-e_7\!=\!-e_6e_1\!=\!(e_3e_2)e_1.
\end{eqnarray*}
Let $X_1=Re_1+Re_2+Re_3$. Then $X_1$ is an additive subgroup of $X$ satisfying
$uu=0$ and $uv+vu=0$ for every $u$, $v\in X_1$. Moreover, we have
$X_2=Re_4+Re_5+Re_6$ and $X_3=Re_7$.

Let $Q$ be the Moufang loop of order $4^7$ constructed from $X$ as in
Proposition \ref{pr:bruckConstruction}. By Proposition \ref{pr:innQchar4},
$\inn(Q)$ is an abelian group of order $4^6$ and exponent $4$. By Proposition
\ref{pr:bruckConstruction}, we have
\begin{displaymath}
    [[(e_1,0,0),(e_2,0,0)],(e_3,0,0)] = (0,0,2e_7) \neq (0,0,0),
\end{displaymath}
which implies that the nilpotency class of $Q$ is at least $3$ (in fact, it is
equal to $3$). Hence $Q$ is a Moufang loop of Cs\"org\H{o} type.

\bibliographystyle{plain}

\begin{thebibliography}{99}

\bibitem{BruckTAMS} R.~H.~Bruck, \emph{Contributions to the theory of loops},
Trans. Amer. Math. Soc. \textbf{60} (1946), 245--354.

\bibitem{Bruck} R.~H.~Bruck, A survey of binary systems, Third printing,
corrected, \emph{Ergebnisse der Mathematik und ihrer Grenzgebiete} \textbf{20},
Springer-Verlag, 1971.

\bibitem{Csorgo} P.~Cs\"org\H o, \emph{Abelian inner mappings and nilpotency class greater than two},
European~J.~Combin. \textbf{28} (2007), 858--868.

\bibitem{CsDr} P.~Cs\"{o}rg\H{o} and A.~Dr\'apal, \emph{Left conjugacy
closed loops of nilpotency class two}, Results Math. \textbf{47} (2005), no.
\textbf{3}-\textbf{4}, 242--265.

\bibitem{DrKi} A.~Dr\'apal and M.~K.~Kinyon,
    \emph{Buchsteiner loops: associators and constructions}, preprint.

\bibitem{DrVo} A.~Dr\'apal and P.~Vojt\v{e}chovsk\'y, \emph{Explicit
    constructions of loops with commuting inner mappings}, European J. Combin.
    \textbf{29} (2008), no. \textbf{7}, 1662--1681.


\bibitem{Hsu} T.~Hsu, \emph{Moufang loops of class $2$ and cubic forms},
Math. Proc. Cambridge Philos. Soc. \textbf{128} (2000), no. \textbf{2}, 197--222.

\bibitem{KePh} T.~Kepka and J.~D.~Phillips, \emph{Connected transversals to
    subnormal subgroups}, Commentationes Mathematicae Universitatis Carolinae
    \textbf{38}, \textbf{2} (1997), 223--230.


\bibitem{Ma} M.~Mazur, \emph{Connected transversals to nilpotent groups},
    J.~Group Theory \textbf{10} (2007), 195--203.

\bibitem{Moufang} R.~Moufang, \emph{Zur Struktur von Alternativk\"{o}rpern}, Math. Ann. \textbf{110} (1935), no. \textbf{1},
416--430.


\bibitem{NiKe2} M.~Niemenmaa and T.~Kepka, \emph{On connected transversals to
    abelian subgroups}, Bull. Austral. Math. Soc. \textbf{49} (1994), no. \textbf{1},
    121--128.

\bibitem{Ni} M.~Niemenmaa, \emph{Finite loops with nilpotent inner mapping
    groups are centrally nilpotent}, to appear in Bull. Austral. Math. Soc.

\end{thebibliography}

\end{document}